\documentclass[12pt]{amsart}
\usepackage{amssymb,amsmath,latexsym}
\usepackage{amsbsy}

\usepackage{arydshln}
\usepackage{amsthm}
\usepackage{amsfonts}
\usepackage[thinlines]{easybmat}

\newtheorem{theorem}{Theorem}[section]

\newtheorem{corollary}[theorem]{Corollary}

\newtheorem{example}[theorem]{Example}

\newtheorem{lemma}[theorem]{Lemma}

\newtheorem{proposition}[theorem]{Proposition}
\newtheorem{remark}[theorem]{Remark}

\numberwithin{equation}{section}

\def\tcalta{\mathcal{T}(\theta, \alpha)}
\def\ata{A^{\theta,\alpha}}
\def\aat{A^{\alpha,\theta}}
\def\nk0t{\|\tilde k_0^\theta\|^{-2}}

\def\kda{K_\alpha}
\def\kdt{K_\theta}
\def\kdta{K_{\frac{\theta}{\alpha}}}

\def\b1{\mathcal{B}_1(\kdt,\kda)}
\def\kdb{K_\beta}

\begin{document}

\title{Asymmetric truncated Toeplitz operators and conjugations}

\author[M. C. C\^amara]{M. Cristina C\^amara}
\address{M. Cristina C\^amara, Center for Mathematical Analysis, Geometry and Dynamical Systems, Mathematics Department, Instituto Superior T\'ecnico, Universidade de Lisboa, Av. Rovisco Pais, 1049--001 Lisboa, Portugal}
\email{cristina.camara@tecnico.ulisboa.pt}

\author[K. Kli\'s--Garlicka]{Kamila Kli\'s--Garlicka}
\address{Kamila Kli\'s-Garlicka, Department of Applied Mathematics, University of Agriculture, ul. Balicka 253c, 30-198 Krak\'ow, Poland}
\email{rmklis@cyfronet.pl}

\author[M. Ptak]{Marek Ptak}
\address{Marek Ptak, Department of Applied Mathematics, University of Agriculture, ul. Balicka 253c, 30-198 Krak\'ow, Poland}
\email{rmptak@cyf-kr.edu.pl}

\thanks{Research of the first author was partially supported by Funda\c{c}\~{a}o para a Ci\^{e}ncia e a Tecnologia (FCT/Portugal), through Project UID/MAT/04459/2013. Research of the second and the third  authors was supported by the Ministry of Science and Higher Education of the Republic of Poland.}

\begin{abstract}
Truncated Toeplitz operators in a model space are C--symmetric with respect
to a natural conjugation in that space. We show that this and another
conjugation associated to an orthogonal decomposition possess unique
properties and we study their relations with asymmetric truncated Toeplitz
operators in terms of C--symmetry. New connections with Hankel operators are
established through this approach.
\end{abstract}
\maketitle
\makeatletter
\renewcommand\@makefnmark%
{\mbox{\textsuperscript{\normalfont\@thefnmark)}}}
\makeatother


\section{Introduction}
Let $\mathcal{H}$ be a complex Hilbert space, and denote by $L(\mathcal{H})$ the algebra of all bounded linear operators on $\mathcal{H}$. A {\it conjugation} on $\mathcal{H}$ is an antilinear involution $C\colon \mathcal{H}\to\mathcal{H}$ such that $\langle Cf,Cg\rangle=\langle g,f\rangle$ for all $f, g\in \mathcal{H}$. Conjugations and their relations with various classes of operators have been studied in Hilbert spaces for many years. A new motivation to study them came from  \cite{GP}, and many interesting results  have recently appeared on this topic \cite{Bess, GP2, LZ, KKBLMP, KKMP, StSt}. In particular, the study of {\it C--symmetric} operators, i.e., operators $A\in L(\mathcal{H})$ such that  $CAC=A^*$, has attracted much attention, with particular emphasis on the case where the underlying Hilbert spaces are model spaces, defined as follows.

Let us denote by $L^2$  the space $L^2(\mathbb{T},m)$, where $\mathbb{T}$ is the unit circle and $m$ is the normalized Lebesgue measure on $\mathbb{T}$, and let $H^2=H^2(\mathbb{D})$ be the Hardy space on the unit disc, identified as usual with a subspace of $L^2$. If $\theta$ is an inner function, i.e., $\theta\in H^\infty$ ($H^\infty=H^\infty(\mathbb{D})$ denotes the space of all bounded analytic functions in $\mathbb{D}$), $|\theta(t)|=1$ a.e. on $\mathbb{T}$, the model space $\kdt$ is defined by $\kdt=H^2\ominus\theta H^2$. It follows from Beurling's theorem that these are the invariant subspaces for the classical backward shift $S^*$. We denote by $P_\theta$ the orthogonal projection  from $L^2$ onto $\kdt$, and by $K^\infty_\theta$ the dense subset of $\kdt$ defined by $\kdt^\infty=\kdt\cap H^\infty$ (\cite{Sarason}).

One of the most important classes of operators on model spaces is that of truncated Toeplitz operators (\cite{Sarason}), which have been widely studied recently (see for example \cite{ BCT, GMR, Sarason}).
For $\varphi\in L^2$, a {\it truncated Toeplitz operator} $A_\varphi^{\theta}$ is defined, for all $f\in\kdt$ such that $\varphi f\in L^2$ (and, in particular, for all $f\in\kdt^\infty$), by $$A_\varphi^\theta f=P_\theta(\varphi f).$$
If this operator is bounded, then it can be uniquely extended to a bounded operator on $\kdt$; in that case we say that $A_\varphi^\theta\in \mathcal{T}(\theta)$.

One can define a conjugation $C_\theta$ in $L^2$,
$C_\theta(f)=\theta \bar z \bar f$ for $f\in L^2$, which  preserves
 the model space $\kdt$ (i.e., $C_\theta P_\theta=P_\theta C_\theta$), and therefore induces a conjugation in $\kdt$, also denoted by $C_\theta$. This conjugation plays an important role in the study of truncated Toeplitz operators. In fact,  the latter are $C_\theta$--symmetric \cite{GP}, i.e., $C_\theta AC_\theta=A^*$ for $A\in\mathcal{T}(\theta)$ or, equivalently, $AC_\theta-C_\theta A^*=0$.

More generally, one can consider {\it asymmetric truncated Toeplitz operators} between two (eventually) different model spaces $\kdt$ and $\kda$, where  $\alpha$ and $\theta$ are nonconstant inner functions. For $\varphi\in L^2$,  we define  \begin{equation}A_\varphi^{\theta,\alpha}\colon \mathcal{D}\subset \kdt\to \kda,\quad\quad
   A_\varphi^{\theta,\alpha} f=P_{\alpha}(\varphi f)\end{equation} with domain %
   $\mathcal{D}=\mathcal{D}( A_\varphi^{\theta,\alpha})=\{ f\in \kdt\colon \varphi f\in L^2\}\supset\kdt^\infty$.
   Again, if this operator is bounded, it has a unique bounded extension to $\kdt$,  $A_{\varphi}^{\theta,\alpha}\colon \kdt\to \kda$, and  the class of all such operators is denoted by $\tcalta$.  Recall after \cite{BCKP} that if $\ata_\varphi\in \mathcal{T}(\theta,\alpha)$, then $(\ata_\varphi)^*=\aat_{\bar\varphi}\in\mathcal{T}(\alpha,\theta)$. Asymmetric truncated Toeplitz operators were studied in \cite{BCKP} in the context of $H^2(\mathbb{D})$, and in \cite{CCJP} in the context of the Hardy space on the upper half-plane  $H^p(\mathbb{C}^+)$ ($1<p<\infty$).


When $\alpha$ divides $\theta$ ($\alpha\leqslant \theta$), i.e., $\frac{\theta}{\alpha}$ is an inner function, then $\kda\subset\kdt$ and we have the orthogonal decomposition $\kdt=\kda\oplus\alpha K_{\frac{\theta}{\alpha}}$. This suggests to define another conjugation in $\kdt$, besides $C_\theta$, denoted by $C_{\alpha,\frac{\theta}{\alpha}}$ and defined by \eqref{cat}. It turns out that these conjugations are unique in the sense that they coincide, on both $\kda$ and $\alpha K_{\frac{\theta}{\alpha}}$, with conjugations  on $L^2$ for which the operator of multiplication by the independent variable, $M_z$, is C--symmetric (Theorem \ref{cta}).

In this paper we investigate the relations of asymmetric truncated Toeplitz operators with these two conjugations and we show that certain identities  of C--symmetric type still hold for these operators when the conjugation  C is one of the above mentioned ones, $C_\theta$ or $C_{\alpha,\frac{\theta}{\alpha}}$ (Theorem \ref{csymmetry}). Moreover, since we no longer have the equality $\ata_\varphi C-C(\ata_\varphi)^*=0$ in general, we study various differences of that type and we show that they can be expressed in terms of Hankel operators.

\section{The actions $\diamond$ and $\boxplus$}
In the following section the letters $\mathcal{H}$, $\mathcal{K}$, with or without indexes, denote
 complex Hilbert spaces. Let $L(\mathcal{H}, \mathcal{K})$ (respectively, $LA(\mathcal{H}, \mathcal{K})$) denote the space of all bounded linear (respectively, antilinear) operators from $\mathcal{H}$ to $\mathcal{K}$. Recall that for $X\in LA(\mathcal{H}, \mathcal{K})$ there is a unique antilinear operator $X^\sharp$, called the {\it antilinear adjoint} of $X$, satisfying the equality
\begin{equation}\label{as}\langle Xf,g\rangle=\overline{\langle f,X^\sharp g\rangle},\end{equation}
for all $f\in \mathcal{H}$, $g\in \mathcal{K}$. It is easy to see that the antilinear adjoint has the following properties:

\begin{proposition}\label{adjoint}
\begin{enumerate}
\item If $X\in LA(\mathcal{H},\mathcal{K})$, then $(X^\sharp)^\sharp=X$.
\item If $X_1\in LA(\mathcal{H}_1,\mathcal{H}_2)$ and  $X_2\in LA(\mathcal{H}_2,\mathcal{H}_3)$,  then $X_2 X_1\in L(\mathcal{H}_1,\mathcal{H}_3)$ and $(X_2 X_1)^*= X_1^\sharp X_2^\sharp$.
\item If  $A\in L(\mathcal{H}_1, \mathcal{H})$ and $X\in LA(\mathcal{H},\mathcal{K})$, then $(XA)^\sharp=A^*X^\sharp$.
\item If  $B\in L(\mathcal{K},\mathcal{K}_1)$ and $X\in LA(\mathcal{H},\mathcal{K})$, then $(BX)^\sharp=X^\sharp B^*$.
\end{enumerate}
\end{proposition}
Let $X_1\colon \mathcal{H}\to \mathcal{K}_1$, $X_2\colon\mathcal{H}\to \mathcal{K}_2$, $Y_1\colon \mathcal{K}_1\to\mathcal{H}$, $Y_2\colon\mathcal{K}_2\to\mathcal{H}$ be (linear or antilinear) operators.
Define the following actions:
$$X_{1}\diamond X_2 \colon \mathcal{H}\to \mathcal{K}_1\oplus \mathcal{K}_2, \, (X_1\diamond X_2)f=X_1f\oplus X_2 f $$
and
$$ Y_1\boxplus Y_2 \colon \mathcal{K}_1\oplus \mathcal{K}_2\to \mathcal{H}, \, (Y_1\boxplus Y_2)(f\oplus g)=Y_1f + Y_2 g.$$
\begin{proposition}\label{prop2} Let $X_1,X_2, Y_1, Y_2$ be antilinear operators, then:
\begin{enumerate}
\item $(X_{1}\diamond X_2)^\sharp=X_1^\sharp\boxplus X_2^\sharp$;
\item $(X_{1}\boxplus  X_2)^\sharp=X_1^\sharp\diamond X_2^\sharp$;
\item if $A\in L(\mathcal{K}_1\oplus \mathcal{K}_2,\mathcal{K})$, then
$
(A(X_{1}\diamond X_2))^\sharp=(X_1^\sharp\boxplus X_2^\sharp)A^*;
$
\item if $B\in L(\mathcal{H}_1, \mathcal{K}_1\oplus \mathcal{K}_2)$, then $
((Y_1\boxplus Y_2)B)^\sharp=B^*(Y_{1}^\sharp\diamond Y_2^\sharp).
$
\end{enumerate}
\end{proposition}
\begin{proof} To show (1) let us take $f\in \mathcal{H}$, $g_1\in \mathcal{K}_1$, $g_2\in \mathcal{K}_2$. Then
\begin{multline*}
\langle (X_{1}\diamond X_2)f, g_1\oplus g_2\rangle=\langle X_1f\oplus X_2f, g_1\oplus g_2\rangle=\langle X_1f,g_1\rangle+\langle X_2f,g_2\rangle=\\ \overline{\langle f,X_1^\sharp g_1\rangle}+\overline{\langle f,X_2^\sharp g_2\rangle}=\overline{\langle f,(X_1^\sharp\boxplus X_2^\sharp)( g_1\oplus g_2)\rangle}.
\end{multline*}
The equalities (2), (3) and (4) follow directly from (1) and Proposition \ref{adjoint}.
\end{proof}
\begin{remark}
Note that the proposition above holds if we change  antilinear operators to linear operators, ${}^\sharp$ to ${}^*$ and vice versa.
\end{remark}
Now let us consider two conjugations $C_1$, $C_2$ on $\mathcal{H}$.
Define the following actions:
\begin{equation}\label{diabox}  C_\diamond=\tfrac{1}{\sqrt{2}}C_{1}\diamond C_2 \colon \mathcal{H}\to \mathcal{H}\oplus \mathcal{H}, \, \text{and }
C_\boxplus=\tfrac{1}{\sqrt{2}} C_1\boxplus C_2 \colon \mathcal{H}\oplus \mathcal{H}\to \mathcal{H}.
\end{equation}
\begin{proposition}\label{pr23}
Let $C_1$, $C_2$ be conjugations  on $\mathcal{H}$. Then
\begin{enumerate}
\item $C_\boxplus\circ C_\diamond\colon \mathcal{H}\to \mathcal{H}$ and $C_\diamond\circ C_\boxplus\colon \mathcal{H}\oplus \mathcal{H}\to \mathcal{H}\oplus \mathcal{H} $ are linear operators;
    \item $C_\boxplus\circ C_\diamond =I_{\mathcal{H}}$;
    \item $(C_\boxplus)^\sharp=C_\diamond$ and  $(C_\diamond)^\sharp=C_\boxplus$;
\item $C_\diamond\circ C_\boxplus = Q$, where $Q$ is an orthogonal projection;
\item $\ker Q=\{C_2 f\oplus -C_1 f : f\in \mathcal{H} \}$;
\item $\operatorname{ran} Q=\{C_{2}f\oplus C_1 f : f\in \mathcal{H} \}$.
\end{enumerate}
\end{proposition}
\begin{proof} The statement (1) is immediate. To prove (2) let us take
 $f\in \mathcal{H}$. Then  we have
$$\tfrac{1}{2}(C_1\boxplus C_2)(C_1\diamond C_2)f=\tfrac{1}{2}(C_1\boxplus C_2)(C_1 f\oplus C_2 f)=\tfrac{1}{2}(C_1^2 f+C_2^2 f)=f.$$
The equalities in (3) follow from Proposition \ref{prop2}. Take now  $f,g\in \mathcal{H}$, then
\begin{equation}\label{projq}
\tfrac{1}{2}(C_1\diamond C_2)(C_1\boxplus C_2)(f\oplus g)=\tfrac{1}{2}((f+C_1C_2g)\oplus(g+C_2C_1 f)).
\end{equation}
Hence
\begin{multline*}
(\tfrac{1}{2}(C_1\diamond C_2)(C_1\boxplus C_2))^2 (f\oplus g)=\\
\tfrac{1}{2}(C_1\diamond C_2)(C_1\boxplus C_2)(\tfrac{1}{2}((f+C_1C_2g)\oplus(g+C_2C_1 f)))=\\
\tfrac{1}{2}((f+C_1C_2g)\oplus(g+C_2C_1 f)).\end{multline*}
So (4) holds and (5) and (6) follow from \eqref{projq}.
\end{proof}

 The next proposition is related to \eqref{sym} in the main theorem of the Section 5.
 \begin{proposition} \label{p2p5} Let $C_1, C_2$ be conjugations in $\mathcal{H}$ and let $C_\boxplus, C_\diamond$ be defined as in \eqref{diabox}. Let
 $A\in L(\mathcal{H})$ be $C_1$--symmetric and $C_2$--symmetric. Then
$$ C_\boxplus(A\oplus A)C_\diamond = A^*.$$
\end{proposition}
Recall  that any unitary operator $U\in L(\mathcal{H})$ is a product of two conjugations $C_1$, $C_2$ (\cite{[GL]}). Moreover, as it was shown in \cite{GPP}, such a unitary operator is both $C_1$ and $C_2$--symmetric. Hence any unitary operator satisfies the assumptions of Proposition \ref{p2p5} for suitable conjugations.
\newpage
\section{Conjugations in model spaces: $C_\theta$ and $C_{\alpha,\frac{\theta}{\alpha}}$}
Let $\alpha$ and $\theta$ be nonconstant inner functions such that $\alpha\leqslant\theta$. Then by \cite[Lemma 5.10]{GMR} the model space $\kdt$ can be decomposed as $\kda\oplus\alpha \kdta$ or $\kdta\oplus \frac{\theta}{\alpha}\kda$. Hence $P_\theta=P_\alpha+
\alpha P_{\frac{\theta}{\alpha}}\bar\alpha$ and $P_\theta=P_{\frac{\theta}{\alpha}}+{\frac{\theta}{\alpha}}
 P_{{\alpha}}{\frac{\bar \theta}{\bar\alpha}}$.

\begin{proposition}[Proposition 2.3, \cite{BCKP}]\label{rozklad}
Let $\alpha, \theta$ be nonconstant inner functions such that $\alpha \leqslant \theta$. If $f_1\in \kda$ and $f_2\in  \kdta$, then

\begin{enumerate}
\item $C_{\theta}(f_1+\alpha f_2)= C_{\frac{\theta}{\alpha}}f_2+\tfrac{\theta}{\alpha}C_{\alpha}f_1$,
\item $C_{\theta}(f_2+ \tfrac{\theta}{\alpha} f_1)=C_{\alpha}f_1+\alpha C_{\frac{\theta}{\alpha}}f_2$.
\end{enumerate}
\end{proposition}
The orthogonal decomposition $\kdt=\kda\oplus\alpha \kdta$  suggests to consider another conjugation $C_{\alpha,\frac{\theta}{\alpha}}$ on $\kdt$ defined as
\begin{equation}\label{cat}
\begin{alignedat}{12}
 C_{\alpha,\frac{\theta}{\alpha}}:=&C_{\alpha}\oplus \alpha C_{\frac{\theta}{\alpha}}\, \bar{\alpha},\\ C_{\alpha,\frac{\theta}{\alpha}}(g_1+\alpha g_2)=&C_{\alpha}g_1+\alpha C_{\frac{\theta}{\alpha}}g_2=\alpha \bar{z}\,\bar{g_1}+\theta \bar{z}\,\bar{g_2}
 \end{alignedat}
 \end{equation} for $g_1\in \kda$, $g_2\in \kdta$. To see that $C_{\theta,\frac{\theta}\alpha}$ is a conjugation it is enough  to show that $C^2_{\theta,\frac{\theta}\alpha}=I_{\kdt}$. Namely,
\begin{multline*}
(C_\alpha\oplus\alpha C_{\frac{\theta}{\alpha}}\bar\alpha)(C_\alpha\oplus\alpha C_{\frac{\theta}{\alpha}}\bar\alpha)=P_\alpha\oplus\alpha C_{\frac{\theta}{\alpha}}\bar\alpha\alpha C_{\frac{\theta}{\alpha}}\bar\alpha=\\P_\alpha\oplus\alpha I_{\kdta}\bar\alpha=P_\alpha+\alpha P_{\frac{\theta}{\alpha}}\bar\alpha=I_{\kdt}.
\end{multline*}

 For any inner function $\theta$ and $\lambda\in \mathbb{D}$, denote

\[
k_\lambda^\theta(z)=\tfrac{1-\overline{\theta (\lambda)}\,\theta(z)}{1-\bar \lambda \,z}\qquad
\text{ and  }\qquad
\tilde k_\lambda^\theta (z)=\tfrac{\theta(z)-\theta(\lambda)}{z-\lambda}\,.
\]
Recall that $k_\lambda^\theta$ are  reproducing kernel functions for the model space $K_\theta$, i.e., $\langle f, k_\lambda^\theta\rangle=f(\lambda)$ for all $f\in K_\theta$. Assume that $\alpha\leqslant\theta$,
the conjugations $C_\theta$ and $C_{\alpha,\frac{\theta}{\alpha}}$ act on reproducing kernel functions $k_\lambda^\theta$ as follows:
\[
\qquad\quad C_\theta\, k_\lambda^\theta=\tilde k_\lambda^\theta
\quad\ \ \quad \quad\text{ and }\qquad\quad
C_{\alpha,\frac{\theta}{\alpha}}\, k_\lambda^\theta=\tilde k_\lambda^\alpha+{\alpha(\lambda)}\,\alpha\,\tilde k_\lambda^\frac{\theta}{\alpha}.\]
We have also the following "reproducing" properties. For any $f\in\kdt$:
\[
\quad \langle f, C_\theta\, k_\lambda^\theta\rangle= \overline{(C_\theta f)(\lambda)}
\quad\ \  \quad\text{ and }\qquad\quad
\langle f, C_{\alpha,\frac{\theta}{\alpha}}\, k_\lambda^\theta\rangle= \overline{(C_{\alpha,\frac{\theta}{\alpha}}f)(\lambda)}.\]
Moreover, $C_{\alpha,\frac{\theta}{\alpha}}C_\theta$ and $C_\theta C_{\alpha,\frac{\theta}{\alpha}}$ are unitary operators
(as compositions of two conjugations, see \cite{GPP}, \cite{[GL]}), which are inverses of each other. More precisely:
\begin{proposition}
Let $\alpha$, $\theta$ be nonconstant inner functions such that $\alpha\leqslant \theta$.  Then
$C_{\alpha,\frac{\theta}{\alpha}}C_\theta \colon \kdt=\kdta\oplus\tfrac{\theta}{\alpha}\kda\to\kdt=\kda\oplus\alpha \kdta$ and $C_\theta C_{\alpha,\frac{\theta}{\alpha}}\colon \kdt=\kda\oplus\alpha \kdta\to\kdt=\kdta\oplus\tfrac{\theta}{\alpha}\kda$ are unitary operators such that
\begin{enumerate}
\item $C_\theta C_{\alpha,\frac{\theta}{\alpha}}=P_{\frac{\theta}{\alpha}}\bar\alpha+\tfrac{\theta}{\alpha}P_\alpha,$
\item $C_{\alpha,\frac{\theta}{\alpha}}C_\theta =P_\alpha\tfrac{\bar\theta}{\bar\alpha}+\alpha P_{\frac{\theta}{\alpha}}.$
\end{enumerate}
\end{proposition}

As a special case of Proposition \ref{pr23} we have:

\begin{proposition}
Let $\alpha$, $\theta$ be nonconstant inner functions such that $\alpha\leqslant \theta$. Define the following actions:
$$C_\diamond=\tfrac{1}{\sqrt{2}}C_{\alpha,\frac{\theta}{\alpha}}\diamond C_\theta \colon \kdt\to \kdt\oplus \kdt, \, (C_{\alpha,\frac{\theta}{\alpha}}\diamond C_\theta)f=C_{\alpha,\frac{\theta}{\alpha}}f\oplus C_\theta f $$
and
$$C_\boxplus=\tfrac{1}{\sqrt{2}} C_{\alpha,\frac{\theta}{\alpha}}\boxplus C_\theta \colon \kdt\oplus \kdt\to \kdt, \, (C_{\alpha,\frac{\theta}{\alpha}}\boxplus C_\theta)(f\oplus g)=C_{\alpha,\frac{\theta}{\alpha}}f + C_\theta g.$$
Then
\begin{enumerate}
\item $C_\boxplus\circ C_\diamond\colon \kdt\to \kdt$ and $C_\diamond\circ C_\boxplus\colon \kdt\oplus \kdt\to \kdt\oplus \kdt $ are linear operators,
\item $C_\boxplus\circ C_\diamond =I_{\kdt}$,
\item $C_\diamond\circ C_\boxplus = Q$, where $Q$ is an orthogonal projection in $\kdt\oplus\kdt$,
\item $\ker Q=\{C_{\alpha,\frac{\theta}{\alpha}}f\oplus -C_\theta f : f\in \kdt \}$,
\item $\operatorname{ran} Q=\{C_{\alpha,\frac{\theta}{\alpha}}f\oplus C_\theta f : f\in \kdt \}$.
\end{enumerate}
\end{proposition}

\section{$M_z$--conjugations in $L^2$}

In this section we will show that the conjugations $C_\theta$ and $C_{\alpha,\frac{\theta}{\alpha}}$ are in a certain sense unique.

Since we are motivated by truncated Toeplitz operators, we will concentrate on conjugations for which the multiplication by the independent variable $M_z$ is $C$--symmetric. Let $J$ denote the complex conjugation in $L^2$, that is $J\colon L^2\to L^2$, $Jf=\bar f$ for $f\in L^2$. For $\varphi\in L^\infty$, denote by $M_\varphi\colon L^2\to L^2$ a multiplication operator $M_\varphi f=\varphi f$, $f\in L^2$.  A conjugation $C$ on $L^2$ will be called an {\it $M$--conjugation} if $M_\varphi C=CM_{\bar \varphi}$ (i.e., $M_\varphi$ is $C$--symmetric) for all $\varphi\in L^\infty$, and $C$ will be called an {\it $M_z$--conjugation} if $M_z C=C M_{\bar z}$.

 The following theorem fully characterizes $M$--conjugations in $L^2$. It also says that in fact the definitions of $M$--conjugation and $M_z$--conjugation are equivalent.
\begin{theorem}\label{c1}
Let $C$ be a conjugation in $L^2$. Then the following are equivalent:
\begin{enumerate}
\item $M_\varphi C=C M_{\bar\varphi}$ for all $\varphi\in L^\infty$ ($C$ is an $M$--conjugation),
\item $M_z C=C M_{\bar z}$ ($C$ is an $M_z$--conjugation),
 \item there is $\psi\in L^\infty$, with $|\psi|=1$, such that $C=M_\psi J$.
 \end{enumerate}
\end{theorem}
\begin{proof} It is enough to show that $(2)\Rightarrow (3)$.
Assume that $CM_z=M_{\bar z}C$. Then $JCM_z=JM_{\bar z}C=M_z JC$. It means that the linear operator $JC$ commutes with $M_z$. By \cite[Theorem 3.2]{RR} $JC=M_{\bar\psi}$ for some $\psi\in L^\infty$. Hence $C=JM_{\bar\psi}=M_{\psi}J$.

Since $C$ is a conjugation, we have $C^2=I_{L^2}$. Therefore for all $f\in L^2$ we have
$$f=C^2 f=M_\psi J M_\psi J f=M_\psi J({\psi\bar f})=|\psi|^2 f,$$
which implies that $|\psi|=1$ a.e.%
%
\end{proof}

Now we study the invariant subspaces of $M_z$--conjugations and their relations with orthogonal decompositions of model spaces.

\begin{theorem}\label{bb} Let $\alpha, \gamma,  \theta$ be inner functions ($\alpha,\theta$ nonconstant) such that $\gamma\alpha\leqslant \theta$.
Let $C$ be a conjugation in $L^2$ such that $M_z C=C M_{\bar z}$. Assume that  $C(\gamma\kda)\subset\kdt$.  Then there is an inner function $\beta$ such that $C=C_\beta$, with $\gamma\alpha\leqslant\beta\leqslant\gamma \theta$.
\end{theorem}
\begin{proof} Recall the standard notation for the reproducing kernel functions at $0$ in $\kda$, namely, $k_0^\alpha=1-\overline{\alpha(0)}{\alpha}$ and $\tilde k_0^\alpha=C_\alpha k_0^\alpha=\bar z(\alpha-\alpha(0))$.
By Theorem \ref{c1} we know that $C=M_\psi J$ for some function $\psi\in L^\infty$, $|\psi|=1$. Hence
$$\kdt \ni C (\gamma\tilde k_0^\alpha)= M_\psi J (\gamma\tilde k_0^\alpha)=\psi \overline{\gamma\bar z(\alpha-\alpha(0))}=\bar\gamma\bar\alpha z\psi(1-\overline{\alpha(0)}\alpha).$$
Thus there is $h\in \kdt$ such that $h=\bar\gamma\bar\alpha z\psi(1-\overline{\alpha(0)}\alpha)$.
Since $(1-\overline{\alpha(0)}\alpha)^{-1}$ is a bounded analytic function, we have
$$\bar \gamma\bar\alpha z\psi=h(1-\overline{\alpha(0)}\alpha)^{-1}\in H^2.$$
Since $\beta_1=\bar\gamma\bar\alpha z\psi\in H^2$ and $|\bar \gamma\bar\alpha z\psi|=1$ a.e. on $\mathbb{T}$, it has to be an inner function.

On the other hand, we have similarly
 $$\kdt\ni C_\theta C (\gamma k_0^\alpha) =C_\theta(\psi \overline{\gamma (1-\overline{\alpha(0)}\alpha)}=\theta  \gamma\bar z\bar\psi (1-\overline{\alpha(0)}\alpha),$$
 and ${\theta}{\gamma} \bar z\bar\psi\in H^2$. Hence \[{H^2}\ni {\theta}{\gamma} \bar z\bar \psi=\tfrac{\theta}{\alpha}\overline{\bar \gamma\,\bar\alpha z\psi}=\tfrac{\theta}{\alpha}\,\overline{\beta_1}.\] But this is only possible when $\beta_1$ divides $\tfrac{\theta}{\alpha} $. Hence $\psi=\gamma\alpha \beta_1 \bar z=\beta\bar z$ with $\gamma\alpha\leqslant\beta\leqslant \gamma\theta$. Finally, we have  $C=C_\beta$.
\end{proof}

Taking $\gamma=1$ and $\alpha=\theta$ we conclude that the conjugation $C_\theta$ is the only $M_z$--conjugation in $L^2$ which preserves the model space $\kdt$.

\begin{theorem}\label{ct}
Let $C$ be an $M_z$--conjugation in $L^2$ (i.e., $M_z C=C M_{\bar z}$). Assume that  $C(\kdt)\subset\kdt$ for some nonconstant inner function $\theta$. Then $C=C_\theta$.
\end{theorem}

\begin{remark}\label{bbb}
Let us consider nonconstant inner functions $\alpha$, $\beta$, $\theta$ such that
$\alpha\leqslant\beta\leqslant \theta$. Then we have the decompositions:
\[\kdt=\kdb\oplus\beta K_{\frac{\theta}\beta}=\kda\oplus\alpha K_{\frac{\beta}\alpha}\oplus\beta K_{\frac{\theta}\beta}.\]
Observe that $C_\beta(\kda)\subset\kdb$. Let $\tilde C$ be any conjugation on $\beta K_{\frac{\theta}\beta}$. Then
$C_\beta{}|_{K_\beta}\oplus \tilde C$ is a conjugation on $\kdt$.
\end{remark}
The following is a consequence of Theorem \ref{bb} and Remark \ref{bbb}.
\begin{proposition}\label{abb}Let $\alpha\leqslant \theta$ be some nonconstant inner functions.
Let $C$ be an $M_z$--conjugation in $L^2$ (i.e., $M_z C=C M_{\bar z}$). Assume that  $C(\kda)\subset\kdt$. Let $\tilde C$ be a conjugation on $\kdt$ such that $C{|_{\kda}}=\tilde C|_{\kda}$. Then there is an inner function $\beta$ with $\alpha\leqslant\beta\leqslant \theta$ and a certain conjugation $\tilde{\tilde C}$ on $\beta K_{\frac{\theta}{\beta}}$ such that  $\tilde C_{\theta}= C_\beta \oplus \tilde{\tilde C}$.
\end{proposition}
The following lemma will be used to prove the next theorem.

\begin{lemma}\label{xxa}
Let $\alpha_1$, $\alpha_2$ be nonconstant inner functions and let $\gamma_1,\gamma_2$ be inner functions such that $\gamma_1\leqslant\alpha_1$ and $\gamma_2\leqslant\alpha_2$. Assume that $\gamma_1\,K_{\alpha_2}\oplus \gamma_2\,K_{\alpha_1}= K_{\alpha_1\alpha_2}$. Then $\gamma_1=1$,   $\gamma_2=\alpha_2$ or $\gamma_1=\alpha_1$,   $\gamma_2=1$.
\end{lemma}
\begin{proof} Recall that inner functions are identified up to multiplication by a constant and let us assume that neither $\gamma_1$ nor $\gamma_2$ is constant.
By \cite[Theorem 5.11]{GMR} we can decompose the space $K_{\alpha_1\alpha_2}$ in two ways
\[K_{\alpha_1\alpha_2}=K_{\gamma_1}\oplus\gamma_1\,K_{\alpha_2}\oplus \gamma_1\alpha_2\,K_{\frac{\alpha_1}{\gamma_1}}=K_{\gamma_2}\oplus\gamma_2\,K_{\alpha_1}\oplus \gamma_2\alpha_1\,K_{\frac{\alpha_2}{\gamma_2}}. \]
Since $\gamma_1\,K_{\alpha_2}\oplus \gamma_2\,K_{\alpha_1}= K_{\alpha_1\alpha_2}$, we have
\[\gamma_1\,K_{\alpha_2}=K_{\gamma_2}\oplus \gamma_2\alpha_1\,K_{\frac{\alpha_2}{\gamma_2}}. \]
Therefore \[K_{\gamma_2}\subset \gamma_1\,K_{\alpha_2}\subset \gamma_1H^2.\]
It follows, as in \cite[Lemma 4.2]{BCKP} that $\gamma_1$ has to be a constant or $K_{\gamma_2}=\{0\}$, i.e., $\gamma_2$ is a constant, and so we obtain a contradiction.

If $\gamma_1=1$, then, by \cite[Theorem 5.11]{GMR}, we have
\[K_{\alpha_2}\oplus \gamma_2\,K_{\alpha_1}= K_{\alpha_1\alpha_2}=K_{\alpha_2}\oplus \alpha_2\,K_{\alpha_1},\]
hence $\gamma_2=\alpha_2$. If $\gamma_1$ is not a constant, then we obtain $\gamma_1=\alpha_1$, $\gamma_2=1$, analogously.
\end{proof}

The definition of the conjugation $C_{\alpha,\frac{\theta}{\alpha}}$ is natural in  view of the orthogonal decomposition  $K_\theta=\kda\oplus\alpha K_{\frac{\theta}{\alpha}}$. However, it is easy to see that $M_z$ is not $C_{\alpha,\frac{\theta}{\alpha}}$--symmetric. Moreover, $C_{\alpha,\frac{\theta}{\alpha}}$  is not a restriction to $\kdt$ of any  $M_z$--conjugation $C$ on $L^2$. On the other hand, the restrictions of $C_{\alpha,\frac{\theta}{\alpha}}$ to the spaces $\kda$ and $\kdt\ominus\kda$ are equal respectively to the restrictions of some (different) $M_z$--conjugations. In the following result we show  that $C_{\alpha,\frac{\theta}{\alpha}}$ and $C_\theta$ are the only conjugations in $\kdt$ with this property.
\begin{theorem}\label{cta}
Let $\alpha$, $\theta$ be nonconstant inner functions such that $\alpha\leqslant\theta$, and let $\tilde C$ be a conjugation on $\kdt$. Assume that there are conjugations $C_i$, $i=1,2$, on $L^2$ with $M_z\,C_i=C_i M_{\bar z}$ such that $\tilde C{|_{\kda}}={C_1}|_{\kda}$ and
 $\tilde C{|_{\kdt\ominus\kda}}={C_2}|_{\kdt\ominus\kda}$. Then $\tilde C=C_\theta$ or $\tilde C=C_{\alpha,\frac{\theta}{\alpha}}=C_\alpha\oplus \alpha C_{\frac{\theta}{\alpha}}\bar\alpha$.
\end{theorem}
\begin{proof} Note firstly that $C_1(\kda)=\tilde C(\kda)\subset \kdt$.
By Theorem \ref{bb} there is an inner function $\gamma_1$, $1\leqslant \gamma_1\leqslant\frac{\theta}{\alpha}$, such that\[\tilde C{|_{\kda}}=C_1{|_{\kda}}=C_{{\gamma_1}{\alpha}}{}|_{\kda}\colon \kda\to\gamma_1\kda\subset\kdt
.\]Recall that $C_{{\gamma_1}{\alpha}}|_{\kda}
f_{\alpha}=\gamma_1\alpha\bar z\bar f_\alpha=\gamma_1C_\alpha f_\alpha$ for $ f_\alpha\in \kda$, and note that $C_{{\gamma_1}{\alpha}}|_{\kda}$ is a bijection between $\kda$ and $\gamma_1\kda$.
 Similarly,  $C_2(\alpha K_{\frac{\theta}{\alpha}})=  C_2(\kdt\ominus\kda)=\tilde C(\kdt\ominus\kda)\subset \kdt$. Hence there is an inner function $\gamma_2$, $1\leqslant \gamma_2\leqslant{\alpha}$, such that
 \[\tilde C|_{\kdt\ominus\kda}=C_2{|_{\alpha K_{\frac{\theta}{\alpha}}}}=C_{{\gamma_2}{\theta}}|_{\alpha K_{\frac{\theta}{\alpha}}}\colon {\alpha K_{\frac{\theta}{\alpha}}}\to\gamma_2{ K_{\frac{\theta}{\alpha}}}\subset\kdt.\]
 On the other hand, \[C_{{\gamma_2}{\theta}}|_{\alpha K_{\frac{\theta}{\alpha}}}\alpha f_{{{\frac{\theta}{\alpha}}}}=C_{{\gamma_2}{\theta}}
(\alpha f_{{{\frac{\theta}{\alpha}}}})=\gamma_2\theta\bar z\bar \alpha \bar f_{\frac{\theta}{\alpha}}=\gamma_2\tfrac{\theta}{\alpha}\bar z \bar f_{\frac{\theta}{\alpha}}=\gamma_2C_{\frac{\theta}{\alpha}} f_{\frac{\theta}{\alpha}}\] for   $f_{\frac{\theta}{\alpha}}\in K_{\frac{\theta}{\alpha}}$. Note that $C_{{\gamma_2}{\theta}}|_{\alpha K_{\frac{\theta}{\alpha}}}$ is a bijection between $\alpha K_{\frac{\theta}{\alpha}}$ and $\gamma_2 K_{\frac{\theta}{\alpha}}$.
Since involution preserves orthogonality and $\kda\oplus \alpha K_{\frac{\theta}{\alpha}}=\kdt$, we get that
$\gamma_1\kda\oplus \gamma_2 K_{\frac{\theta}{\alpha}}=\kdt$. By Lemma \ref{xxa} there are now only two possibilities: either  $\gamma_1=1, \gamma_2=\alpha$ or  $\gamma_1=\frac{\theta}{\alpha}, \gamma_2=1$. In the second case
$\tilde C{|_{\kda}}=C_{{\theta}}|_{\kda} $ and $\tilde C{|_{\kdt\ominus\kda}=C_\theta|_{\alpha K_{\frac{\theta}{\alpha}}}}$, hence $\tilde C=C_\theta$. In the first case
$\tilde C{|_{\kda}}=C_{{\alpha}}|_{\kda} $ and $\tilde C{|_{\alpha K_{\frac{\theta}{\alpha}}}}=C_{{\alpha}{\theta}}|_{\alpha K_{\frac{\theta}{\alpha}}}$, since for  $f_{\frac{\theta}{\alpha}}\in K_{\frac{\theta}{\alpha}}$ we have
\[C_{{\alpha}{\theta}}|_{\alpha K_{\frac{\theta}{\alpha}}}\alpha f_{\frac{\theta}{\alpha}}=\alpha\theta\bar z\bar\alpha\bar f_{\frac{\theta}{\alpha}}=\alpha C_{\frac{\theta}{\alpha}}f_{\frac{\theta}{\alpha}}
=\alpha C_{\frac{\theta}{\alpha}}\bar \alpha \alpha f_{\frac{\theta}{\alpha}}=\alpha C_{\frac{\theta}{\alpha}}\bar \alpha |_{\alpha K_{\frac{\theta}{\alpha}}} \alpha f_{\frac{\theta}{\alpha}}.
\] Hence
$\tilde C=C_\alpha\oplus \alpha C_{\frac{\theta}{\alpha}}\bar\alpha=C_{\alpha,\frac{\theta}{\alpha}}$.
\end{proof}
\begin{example}
Let $\theta=z^5$ and $\alpha =z^3$.
The only conjugation,
besides  $C_{z^5}$, defined by
$C_{z^5}(z_0,z_1,z_2, z_3, z_4)=(\bar z_4,\bar z_3, \bar z_2,\bar z_1, \bar z_0)$, fulfilling the conditions of Theorem \ref{cta} is the conjugation
$C_{z^3,z^2}$ given by $C_{z^3,z^2}(z_0,z_1,z_2, z_3, z_4)=(\bar z_2,\bar z_1, \bar z_0,\bar z_4, \bar z_3)$.
\end{example}

\begin{example}
Let $\theta(z)=\exp\tfrac{z+1}{z-1}$ and $\alpha(z) =\exp(a\tfrac{z+1}{z-1})$ for \mbox{$0<a<1$.} Then $\tfrac{\theta}{\alpha}(z)=\exp((1-a)\tfrac{z+1}{z-1})$.
The only conjugation,
besides  $C_{\theta}$ (defined by
$C_{\theta}(f)=\theta \bar z\bar f$ for $f\in K_\theta$), fulfilling the conditions of Theorem \ref{cta} is the conjugation
$C_{\alpha,\frac{\theta}{\alpha}}$ given by $C_{\alpha,\frac{\theta}{\alpha}}(f_{\alpha}\oplus \alpha{f_{\frac\theta\alpha}})=\alpha\bar z\bar f_\alpha
+\theta \bar z\bar f_{\frac\theta\alpha}$ for $f_{\alpha}\oplus \alpha{f_{\frac\theta\alpha}}\in K_\theta=K_\alpha\oplus
\alpha K_{\frac{\theta}{\alpha}}$.
\end{example}

\section{C-symmetry of asymmetric truncated Toeplitz operators}
Let $C\colon \mathcal{H}\to\mathcal{H}$ be a conjugation. Note that every conjugation is antilinearly selfadjoint, i.e., $C^\sharp=C$.
The next lemma gives simple but important equivalent conditions for an operator to be $C$--symmetric.
 \begin{lemma}\label{antself} Let $A\in L(\mathcal{H})$. Then the following are equivalent:
 \begin{enumerate}
 \item $A$ is $C$--symmetric;
 \item $AC$ is antilinearly selfadjoint, i.e., $(AC)^\sharp =AC$;
 \item $CA$ is antilinearly selfadjoint, i.e., $(CA)^\sharp =CA$.
 \end{enumerate}
 \end{lemma}

It is well known that truncated Toeplitz operators are $C_\theta$--symmetric, \cite{GP}, i.e., for $A^\theta_\varphi\in\mathcal{T}(\theta)$ we have
\begin{equation}\label{11} A^\theta_\varphi C_\theta=C_\theta A^\theta_{\bar\varphi}.\end{equation}
One may wonder whether $A_\varphi^\theta$ is $C_{\alpha,\frac{\theta}{\alpha}}$--symmetric for all $\alpha\leqslant \theta$  but that is not the case in general, as it is shown by this simple example.
\begin{example}
Let 
 $\theta =z^2$ and $\alpha=z$. Then $C=C_{z,z}=J$ is the conjugation given by $C(z_0,z_1)=(\bar z_0, \bar z_1)$, $(z_0,z_1)\in \mathbb{C}^2$. Take a Toeplitz matrix
$A=\left[
   \begin{BMAT}{cc}{cc}
     0 & -1  \\
     1 & \phantom{-}0
   \end{BMAT}
 \right]
$. Then $AC\neq CA^*$.
\end{example}%
Since an asymmetric truncated Toeplitz operator $\ata_\varphi\in\mathcal{T}(\theta,\alpha)$ reduces to the truncated Toeplitz operator $A_\varphi^\theta$, and $C_{\alpha,\frac{\theta}{\alpha}}=C_\theta$ when $\alpha=\theta$, we may ask whether the following generalizations of \eqref{11} hold:
\begin{align}\ata_\varphi C_\theta&=C_\theta\aat_{\bar\varphi}P_\alpha \qquad\quad\text{or}\label{defekt1}\\
\ata_\varphi C_{\alpha,\frac{\theta}{\alpha}}&=C_{\alpha,\frac{\theta}{\alpha}}\aat_{\bar\varphi}P_\alpha.\label{defekt2}
\end{align}

It is easy to see that neither \eqref{defekt1} nor \eqref{defekt2} are true in general. To obtain properties which, in the context of  Lemma \ref{antself}, can be regarded in some sense as describing $C$--symmetric properties of truncated Toeplitz operators,
we will consider the whole space $\kdt$ and use the actions $\diamond$, $\boxplus$.

\begin{theorem}\label{csymmetry}
Let $\alpha$, $\theta$ be nonconstant inner functions such that $\alpha\leqslant \theta$, and let $\varphi\in L^2$ be such that all asymmetric truncated Toeplitz operators below are bounded.
Let us consider the conjugations $C_\theta$ and $C_{\alpha,\frac{\theta}{\alpha}}=C_{\alpha}\oplus \alpha C_{\frac{\theta}{\alpha}} \bar{\alpha}$ in $\kdt=\kda\oplus \alpha \kdta$.
Then the following equalities hold:
\begin{align}&(\ata_\varphi \diamond\alpha A^{\theta,\frac{\theta}{\alpha}}_{\varphi\bar\alpha})C_\theta=C_\theta(\aat_{\bar\varphi}\boxplus A^{\frac{\theta}{\alpha},\theta}_{\bar\varphi\alpha}\bar\alpha), \label{symm1}\\
&(\ata_\varphi \diamond\alpha A^{\theta,\frac{\theta}{\alpha}}_{\varphi\frac{\theta}{\alpha}})C_{\alpha,\frac{\theta}{\alpha}}=C_{\alpha,\frac{\theta}{\alpha}}(\aat_{\bar\varphi}\boxplus A^{\frac{\theta}{\alpha},\theta}_{\overline{\varphi\frac{\theta}{\alpha}}}\bar\alpha),\label{symm2}\\
&(A_{\varphi}^{\theta, \alpha} \oplus \alpha\, A_{\varphi}^{\theta, \frac{\theta}{\alpha}} )( C_{\alpha,\frac{\theta}{\alpha}}\diamond C_\theta)=(C_{\alpha,\frac{\theta}{\alpha}} {\boxplus} C_\theta)(A_{\bar{\varphi}}^{ \alpha,\theta} \oplus  A_{\bar{\varphi}}^{\frac{\theta}{\alpha},\theta } \bar{\alpha}).\label{sym}
\end{align}
    Equivalently, the above operators are antilinearly selfadjoint, i.e.,
\begin{align}&((\ata_\varphi \diamond\alpha A^{\theta,\frac{\theta}{\alpha}}_{\varphi\bar\alpha})C_\theta)^\sharp=
(\ata_\varphi \diamond\alpha A^{\theta,\frac{\theta}{\alpha}}_{\varphi\bar\alpha})C_\theta,\tag{\ref{symm1}a}\label{sym1i}\\
&((\ata_\varphi \diamond\alpha A^{\theta,\frac{\theta}{\alpha}}_{\varphi\frac{\theta}{\alpha}})C_{\alpha,\frac{\theta}{\alpha}})^\sharp=
(\ata_\varphi \diamond\alpha A^{\theta,\frac{\theta}{\alpha}}_{\varphi\frac{\theta}{\alpha}})C_{\alpha,\frac{\theta}{\alpha}},\tag{\ref{symm2}a}\label{sym2i}\\
&((A_{\varphi}^{\theta, \alpha} \oplus \alpha\, A_{\varphi}^{\theta, \frac{\theta}{\alpha}} )( C_{\alpha,\frac{\theta}{\alpha}}\diamond C_\theta))^\sharp =(A_{\varphi}^{\theta, \alpha} \oplus \alpha\, A_{\varphi}^{\theta, \frac{\theta}{\alpha}} )( C_{\alpha,\frac{\theta}{\alpha}}\diamond C_\theta).\tag{\ref{sym}a}\label{sadj}
\end{align}
\end{theorem}%
\begin{proof}
Let us take $f_1\in K^\infty_\alpha$, $f_2\in K^\infty_{\frac{\theta}{\alpha}}$ and
 $f=f_1\oplus \alpha f_2$ (recall that $K^\infty_\alpha\oplus\alpha K^\infty_{\frac{\theta}{\alpha}} $ is dense in $\kdt$ -- see \cite[(5.23)]{GMR}).
 To prove  \eqref{symm1} note that 
\begin{multline}\label{aab}
(\ata_\varphi\diamond \alpha A^{\theta,\frac{\theta}{\alpha}}_{\varphi\bar\alpha})C_\theta f=P_\alpha(\varphi C_\theta f)+\alpha P_{\frac{\theta}{\alpha}}\bar\alpha(\varphi C_\theta f)=\\P_\theta(\varphi C_\theta f)=P_\theta C_\theta(\bar\varphi f)=C_\theta P_\theta(\bar\varphi f),\quad
\end{multline}
since $P_\alpha+\alpha P_{\frac{\theta}{\alpha}}\bar\alpha=P_\theta$. On the other hand, 
we obtain
\begin{align*}
C_\theta(\aat_{\bar\varphi}\boxplus A^{\frac{\theta}{\alpha},\theta}_{\bar\varphi\alpha}\bar\alpha)(f_1\oplus \alpha f_2) =C_\theta(P_\theta(\bar\varphi f_1)+P_\theta(\bar\varphi\alpha f_2))
= C_\theta P_\theta(\bar\varphi f).
\end{align*}

Now we will show that
$$C_{\alpha,\frac{\theta}{\alpha}}(\ata_\varphi \diamond\alpha A^{\theta,\frac{\theta}{\alpha}}_{\varphi\frac{\theta}{\alpha}})=(\aat_{\bar\varphi}\boxplus A^{\frac{\theta}{\alpha},\theta}_{\overline{\varphi\frac{\theta}{\alpha}}}\bar\alpha)
C_{\alpha,\frac{\theta}{\alpha}},$$which is equivalent to  \eqref{symm2}.  Note that
\begin{align*}
C_{\alpha,\frac{\theta}{\alpha}}(\ata_\varphi \diamond &\alpha A^{\theta,\frac{\theta}{\alpha}}_{\varphi\frac{\theta}{\alpha}})f=C_{\alpha,\frac{\theta}{\alpha}}P_\alpha (\varphi f)+C_{\alpha,\frac{\theta}{\alpha}}(\alpha P_{\frac{\theta}{\alpha}}(\tfrac{\theta}{\alpha}\varphi f))\\=&C_\alpha P_\alpha(\varphi f)+\alpha C_{\frac{\theta}{\alpha}}P_{\frac{\theta}{\alpha}}(\tfrac{\theta}{\alpha}\varphi f)
=P_\alpha(\bar\varphi C_\alpha f)+\alpha P_{\frac{\theta}{\alpha}}\bar\alpha(\bar\varphi C_\alpha f)\\=& P_\theta (\bar\varphi C_\alpha f).
\end{align*}
On the other hand, 
$C_\alpha f=\bar z(\alpha \bar f_1+\bar f_2)$. Hence by Proposition \ref{rozklad}  we get
\begin{align*}
(\aat_{\bar\varphi}\boxplus &A^{\frac{\theta}{\alpha},\theta}_{\overline{\varphi\frac{\theta}{\alpha}}}\bar\alpha)C_{\alpha,\frac{\theta}{\alpha}}(f_1\oplus\alpha f_2)=
(\aat_{\bar\varphi}\boxplus A^{\frac{\theta}{\alpha},\theta}_{\overline{\varphi\frac{\theta}{\alpha}}}\bar\alpha)(C_{\alpha}f_1\oplus\alpha C_{\frac{\theta}{\alpha}}f_2)\\
=&\  P_\theta(\bar\varphi C_\alpha f_1)+P_\theta (\bar\varphi\tfrac{\bar\theta}{\bar\alpha}C_{\frac{\theta}{\alpha}}f_2)=P_\theta(\bar\varphi(\alpha \bar z\bar f_1+\bar z \bar f_2))=P_\theta(\bar\varphi C_\alpha f).
\end{align*}

To prove  \eqref{sym}, since $P_\theta=P_\alpha+
\alpha P_{\frac{\theta}{\alpha}}\bar\alpha$, note that
 \begin{align*}
(A_{\varphi}^{\theta, \alpha} \oplus \alpha\,& A_{\varphi}^{\theta, \frac{\theta}{\alpha}} )((C_{\alpha}f_1 + \alpha\, C_{\frac{\theta}{\alpha}}f_2) \oplus C_{\theta }f)\\ = &\
P_{\alpha}(\varphi (C_{\alpha}f_1 + \alpha\, C_{\frac{\theta}{\alpha}}f_2)) + \alpha\, P_{\frac{\theta}{\alpha}}(\varphi C_{\theta} f) \\
=&\
 P_{\alpha} (\varphi \alpha\bar{z} \bar{f_1}+ \theta \varphi \bar{z} \bar{f_2}) +\alpha P_{\frac{\theta}{\alpha}} (\varphi \theta\bar{z}\bar{f_1}+ \varphi\theta \bar{z} \bar{\alpha} \bar{f_2}) \\=&\
 P_\alpha(C_\alpha(\bar\varphi f_1))+P_\alpha(C_{\frac{\theta}{\alpha}}(\bar\varphi \bar{\alpha} f_2))  +   \alpha P_{\frac{\theta}{\alpha}}(\bar\alpha C_{\frac{\theta}{\alpha}}(\bar\varphi \bar{\alpha} f_2))\\ &+ \alpha P_{\frac{\theta}{\alpha}}(C_{\frac{\theta}{\alpha}}(\bar\varphi \bar{\alpha}f_1))\\
 =&\  P_\alpha(C_\alpha(\bar\varphi f_1)) + \alpha P_{\frac{\theta}{\alpha}}(C_{\frac{\theta}{\alpha}}(\bar\varphi \bar{\alpha}f_1))+P_\theta(C_{\frac{\theta}{\alpha}}(\bar\varphi \bar{\alpha} f_2)).
\end{align*}
On the other hand,
\begin{align*}
 (C_{\alpha,\frac{\theta}{\alpha}}\boxplus C_\theta)&(A_{\bar{\varphi}}^{ \alpha,\theta} \oplus  A_{\bar{\varphi}}^{\frac{\theta}{\alpha},\theta } \bar{\alpha})(f_1\oplus \alpha f_2)
\\&=
 C_{\alpha}(P_{\alpha} (\bar{\varphi }f_1))+\alpha C_{\frac{\theta}{\alpha}} (P_{\frac{\theta}{\alpha}} (\bar\alpha\bar{\varphi }f_1))+C_\theta (P_\theta(\bar\varphi f_2)).
\end{align*}
Using  $P_\theta=P_{\frac{\theta}{\alpha}}+{\frac{\theta}{\alpha}}
 P_{{\alpha}}{\frac{\bar \theta}{\bar\alpha}}$ we obtain
\begin{align*}
C_\theta (P_\theta(\bar\varphi f_2))&=C_\theta(P_\alpha(\bar\varphi f_2))+\alpha P_{\frac{\theta}{\alpha}}(\bar\alpha\bar\varphi f_2)\\
&=C_{\frac{\theta}{\alpha}}(P_{\frac{\theta}{\alpha}}(\bar\alpha\bar\varphi f_2))+\tfrac{\theta}{\alpha} C_\alpha(P_\alpha(\bar\varphi f_2)) \\
&= P_{\frac{\theta}{\alpha}}(C_{\frac{\theta}{\alpha}}(\bar\alpha\bar\varphi f_2))+\tfrac{\theta}{\alpha}P_\alpha( C_\alpha(\bar\varphi f_2)) \\
&= P_{\frac{\theta}{\alpha}}(C_{\frac{\theta}{\alpha}}(\bar\alpha\bar\varphi f_2))+\tfrac{\theta}{\alpha}P_\alpha\tfrac{\bar\theta}{\bar\alpha}( C_{\frac{\theta}{\alpha}}(\bar\alpha\bar\varphi f_2)) = P_\theta(C_{\frac{\theta}{\alpha}}(\bar\alpha\bar\varphi f_2)).
\end{align*} That completes the proof of \eqref{sym}.
All calculations were made on a dense subset of $\kdt$, hence we get all the equalities in the theorem.
\end{proof}

One can also ask for which symbols $\varphi\in L^2$ the equalities \eqref{defekt1} and \eqref{defekt2} hold.
From Theorem \ref{csymmetry} and \cite[Theorem 4.4]{BCKP} we obtain the following:
\begin{corollary}\label{wn1}
Let $\alpha$, $\theta$ be nonconstant inner functions such that $\alpha\leqslant \theta$, and let $A\in\mathcal{T}(\theta,\alpha)$. Then
\begin{enumerate}
\item $A C_\theta=C_\theta A^*P_\alpha$ if and only if there is $\varphi\in \overline{\frac\theta\alpha K_\alpha}$ such that $A=\ata_\varphi$,
\item $A C_{\alpha,\frac{\theta}{\alpha}}=C_{\alpha,\frac{\theta}{\alpha}}A^*P_\alpha$ if and only if there is $\varphi\in \kda$ such that $A=\ata_\varphi$.
\end{enumerate}
\end{corollary}
\begin{proof}
Note that to obtain the desired equality (1) we have to assume that $A^{\theta,\frac{\theta}{\alpha}}_{\varphi\bar\alpha}=0$ in the formula \eqref{symm1} of Theorem \ref{csymmetry}, which is equivalent by \cite[Theorem 4.4]{BCKP} to $\varphi\bar\alpha\in \frac{\theta}{\alpha}H^2+\overline{\theta H^2}$, i.e., $\varphi\in \theta H^2+\overline{\frac{\theta}{\alpha}H^2}$. Since for $\varphi \in \alpha H^2+\overline{\theta H^2}$ the operator $\ata_\varphi=0$, we may assume that $\varphi\in \overline{\kdt\cap\tfrac{\theta}{\alpha}H^2}=\overline{\frac\theta\alpha K_\alpha}$.

Similarly, the assumption  $A^{\theta,\frac{\theta}{\alpha}}_{\varphi\frac{\theta}{\alpha}}=0$ is equivalent to $\varphi\frac{\theta}{\alpha}\in \frac{\theta}{\alpha}H^2+\overline{\theta H^2}$. Since for $\varphi \in \alpha H^2+\overline{\theta H^2}$ the operator $\ata_\varphi=0$, it is enough to consider $\varphi\in \kda$ for the equality (2).
\end{proof}

Note that if  $\varphi\in \overline{\frac\theta\alpha K_\alpha}$, then $\ata_\varphi f=P_\alpha\varphi P_\theta f = P_\theta\varphi P_\theta f$ for all $f\in \kdt$, while if $\varphi\in K_\alpha$, then $\ata_\varphi f=P_\alpha\varphi P_\theta f = P_\alpha\varphi P_\alpha f$ for all $f\in \kdt$. Therefore the conditions in (1) and (2) of the previous corollary are satisfied if and only if $\ata_\varphi$ can be identified with  truncated Toeplitz operators  $A^\theta_\varphi$ and $A^\alpha_\varphi$, respectively.

\section{Example with $\theta=z^N$.}
To illustrate the equalities in Theorem \ref{csymmetry} we consider the simplest inner function $\theta=z^N$.
 Then $K_{z^N}$ is the space of polynomials of degree smaller than $N$. Hence $K_{z^N}$ can be identified with $\mathbb{C}^N$. Then the conjugation $C_{z^N}$ in $\mathbb{C}^N$ is given by $C_{z^N}(z_0,\dots, z_N)=(\bar z_N,\dots,\bar z_0)$. Let us
firstly illustrate Lemma \ref{antself}.
   \begin{remark}\label{rem2}
 Let $A\in L(\mathbb{C}^N)$ be a truncated Toeplitz operator with matrix
 $A=(a_{ij})_{i,j=0}^{N-1}$, $a_{ij}=t_{i-j}$ for $i,j=0,\dots,N$. Recall that $A$ is $C_{z^N}$--symmetric, i.e., the matrix is symmetric according to the second diagonal (see \cite{GP}). On the other hand,
 by \eqref{as}, an antilinear  operator $X$ given by a matrix
 $(s_{ij})_{i,j=0}^{N-1}$ is antilinearly selfadjoint if its matrix is symmetric, i.e., $s_{ij}=s_{ji}$ for $i,j=0,\dots,N$. Note that the antilinear operator $AC_{z^N}$ has the Hankel matrix $(b_{ij})_{i,j=0,\dots,N}$, with $b_{i,j}=t_{i+j-N+1}$, which is clearly symmetric ($b_{ij}=b_{ji}$ for $i,j=0,\dots,N$).
 \end{remark}
 Now we will illustrate the equations  \eqref{sym1i}, \eqref{sym2i}, \eqref{sadj}.
\begin{example}
Let $\alpha =z^3$ and $\theta=z^5$. Then any operator in $\mathcal{T}(z^5,z^3)$ has a symbol $\varphi=\sum\limits_{n=-4}^2 a_k z^k\in K_{z^3}+\overline{K_{z^5}}$ (see \cite[Corollary 4.5]{BCKP}). Thus it
has  a matrix representation
$A_\varphi^{z^5,z^3}=\left[   \begin{BMAT}{ccccc}{ccc}
     a_0 & a_{-1} & a_{-2} & a_{-3}& a_{-4}\\
     a_1 & a_0 & a_{-1} & a_{-2}& a_{-3}\\
     a_2 & a_1 & a_0 & a_{-1} & a_{-2}\\
   \end{BMAT} \right]$. To illustrate the  equality \eqref{sym1i} in Theorem \ref{csymmetry} note that $A_{\bar \alpha\varphi}^{z^5,z^2}=\left[   \begin{BMAT}{ccccc}{cc}
     0 & a_2 & a_1 &a_0 & a_{-1}\\ 0 & 0 & a_2& a_1 & a_0\\
       \end{BMAT} \right]$, so $A^{z^5,z^3}_\varphi \diamond z^3 A^{z^5,z^2}_{\varphi\bar z^3}$  is simply the Toeplitz matrix in $\mathbb{C}^5$ with the symbol
       $\varphi=\sum\limits_{n=-4}^2 a_k z^k\in K_{z^3}+\overline{K_{z^5}}\subsetneqq K_{z^5}+\overline{K_{z^5}}$, and its $C_{z^5}$--symmetry or the symmetry of the Hankel matrix $(A^{z^5,z^3}_\varphi \diamond z^3 A^{z^5,z^2}_{\varphi\bar z^3})\,C_{z^5}$ is easily satisfied. Now to obtain equality (1) in Corollary \ref{wn1} in our case we have to assume that $\varphi=a_{-4}\bar z^4+a_{-3}\bar z^3+a_{-2}\bar z^2$, so $a_{-1}=a_0=a_1=a_2=0$.

To illustrate  \eqref{sym2i}, besides the involution $C_{z^5}$, we consider  another involution
$C_{z^3,z^2}(z_0,z_1,z_2, z_3, z_4)=(\bar z_2,\bar z_1, \bar z_0,\bar z_4, \bar z_3)$. Note that $A^{z^5,z^2}_{\varphi z^2}=\left[   \begin{BMAT}{ccccc}{cc}
     a_{-2} & a_{-3} & a_{-4} & 0 & 0\\ a_{-1} & a_{-2} & a_{-3}& a_{-4} & 0\\
       \end{BMAT}\right]$.
Hence
\begin{multline}\label{r11}(A^{z^5,z^3}_\varphi \diamond\alpha A^{z^5,z^2}_{\varphi z^2})C_{z^3,z^2}(z_0,z_1,z_2, z_3, z_4)\\=
\left[   \begin{array}{ccc;{2pt/2pt}cc}
    a_{-2} & a_{-1} & a_{0} & a_{-4} & a_{-3} \\
    a_{-1} & a_0 & a_1 & a_{-3} & a_{-2} \\ a_0 & a_1 & a_2& a_{-2} & a_{-1}\\
     \hdashline[2pt/2pt]
    a_{-4} & a_{-3}& a_{-2} & 0 & 0\\ a_{-3} & a_{-2} & a_{-1} & 0 &a_{-4}\\
   \end{array} \right] \left[ \! \begin{array}{c} \bar z_0 \\ \bar z_1\\ \bar z_2\\ \bar z_3\\ \bar z_4\end{array} \! \right].\end{multline}
Note that to obtain the equality (2) in Corollary \ref{wn1} we have to take $\varphi=a_0+a_1 z+a_2z^2.$

In the equality \eqref{sadj} $A_\varphi^{z^5,z^2}=\left[\! \begin{array}{ccccc}
     a_0 & a_{-1} & a_{-2} &a_{-3} & a_{-4}\\ a_1 & a_0 & a_{-1}& a_{-2} & a_{-3}\\
       \end{array}\! \right]$. Hence
\begin{multline}\label{prz1}
(A_\varphi^{z^5,z^3}\oplus A_\varphi^{z^5,z^2})(C_{z^3,z^2}\diamond C_{z^5})(z_0,z_1,z_2, z_3, z_4)\\=
\left[   \begin{array}{ccc;{2pt/2pt}cc}
    a_{-2} & a_{-1} & a_{0} & a_{-4} & a_{-3} \\
    a_{-1} & a_0 & a_1 & a_{-3} & a_{-2} \\ a_0 & a_1 & a_2& a_{-2} & a_{-1}\\ \hdashline[2pt/2pt] a_{-4} & a_{-3}& a_{-2} & a_{-1} & a_0\\ a_{-3} & a_{-2} & a_{-1} & a_0 &a_1\\
   \end{array} \right] \left[\!\begin{array}{c} \bar z_0 \\ \bar z_1\\ \bar z_2\\ \bar z_3\\ \bar z_4\end{array} \! \right].
\end{multline}
 The equations \eqref{sym2i} and \eqref{sadj} imply  that the antilinear operators $(A^{z^5,z^3}_\varphi \diamond\alpha A^{z^5,z^2}_{\varphi z^2})C_{z^3,z^2}$ and
$(A_\varphi^{z^5,z^3}\oplus A_\varphi^{z^5,z^2})(C_{z^3,z^2}\diamond C_{z^5})$ are antilinearly selfadjoint.
If we write, in both cases,  the above matrices by blocks $\left[\begin{array}{c;{2pt/2pt}c}H_{11}& H_{12}\\  \hdashline[2pt/2pt]H_{21}& H_{22} \end{array}\right]$, then each block is a Hankel matrix and the whole matrix is symmetric, moreover, $H_{12}$ is symmetric to $H_{21}$. In the first case some part of $H_{22}$ annihilates. The above should be also seen in the context of Remark \ref{rem2}.
\end{example}

\newpage
\section{Connections with Hankel operators}
 In light of Theorem \ref{csymmetry} it is natural to ask about the differences $\ata_\varphi C_\theta-C_\theta\aat_{\bar\varphi}P_\alpha$ and
$\ata_\varphi C_{\alpha,\frac{\theta}{\alpha}}-C_{\alpha,\frac{\theta}{\alpha}}\aat_{\bar\varphi}P_\alpha$, which have to become zero when $\alpha=\theta$. It turns out that these differences can be expressed in terms of certain Hankel operators.

Let $P$ denote the orthogonal projection from $L^2$ onto $H^2$, and $P^-$ denote the orthogonal projection from $L^2$ onto $\overline{H^2_0}=L^2\ominus H^2$.
For $\varphi\in L^2$ we define: \[H_{\varphi} \colon H^2 \to \overline{H^2_0}, \quad H_{\varphi} f=P^-(\varphi f);\] for $f\in H^2$ such that $\varphi f\in L^2$. Similarly, for $\theta\in L^\infty$,
\[\widetilde{H}_{\theta} \colon \overline{H^2_0} \to  H^2 , \quad \widetilde{H}_{\theta} f=P(\theta f)\text{ for } f\in \overline{H^2_0}.\]

Let $\theta$ be a nonconstant inner function. Recall firstly the following:
\begin{proposition}\label{proj}
 Let $\theta$ be a nonconstant inner function and let $\kdt=H^2\ominus \theta H^2$ be the associated model space.  Then
\begin{enumerate}
\item $P_{\theta}=\theta P^-\, \bar{\theta}P=\theta P^-\, \bar{\theta}-P^-$,
\item $P_{\theta}f=\theta P^-\, \bar{\theta}f=f-\theta P\bar\theta f$ for all $f \in H^2$,
\item $P_\theta \bar f=P_\theta P \bar f=\overline{f(0)}P_\theta1=\overline{f(0)}(1-\overline{\theta(0)}\theta)$ for all $f \in H^2$.
\end{enumerate}
\end{proposition}
Using Proposition \ref{proj} it it easy to see that, for $A^\theta_\varphi \in \mathcal{T}(\theta)$, both $A^\theta_\varphi C_\theta$ and $C_\theta A^\theta_\varphi$ can be expressed in terms of Hankel operators. In fact we have
\[
A^\theta_\varphi C_\theta =\widetilde{H}_\theta H_{\bar\theta\varphi} C_\theta\;\;\text{and}\;\;C_\theta A^\theta_\varphi =\widetilde{H}_\theta H_{\bar\theta\bar\varphi} C_\theta\,,
\]
which is another way to see that $A^\theta_\varphi C_\theta = C_\theta A^\theta_{\bar\varphi }$, i.e., $ A^\theta_\varphi $ is $C_\theta$-symmetric.

In the asymmetric case ($\alpha <\theta$) we no longer have, in general, either
\begin{align}\ata_\varphi C_\theta&=C_\theta\aat_{\bar\varphi} \qquad\quad\text{or}\label{H1}\\
\ata_\varphi C_{\alpha,\frac{\theta}{\alpha}}&=C_{\alpha,\frac{\theta}{\alpha}}\aat_{\bar\varphi},\label{H2}
\end{align}
where, for simplicity, we identify $ \ata_\varphi$ and $\aat_{\bar\varphi}$ with the operators $P_\alpha\varphi P_\theta$ and $P_\theta{\bar\varphi} P_\alpha$, respectively. Thus it is natural to ask about the differences between the operators on the left and on the right hand sides of the equalities \eqref{H1} and \eqref{H2}. In the following theorem we characterize those differences in terms of Hankel operators. This will later provide, in particular, another way to prove \eqref{sym}.

\begin{theorem}\label{h1} Let $\alpha$, $\theta$ be nonconstant inner functions and $\alpha\leqslant \theta$. If $A_{\varphi}^{\theta, \alpha}\in\mathcal{T}(\theta,\alpha)$ for $\varphi\in L^2$, then the following equalities hold:
\begin{align}
&(\ata_\varphi C_\theta-C_\theta\aat_{\bar\varphi}P_\alpha)f=
    (\widetilde{H}_{\alpha}H_{\bar\alpha\varphi}C_{\frac{\theta}{\alpha}}P_{\frac{\theta}{\alpha}}\bar\alpha-\alpha\widetilde{H}_{\frac{\theta}{\alpha}}H_{\bar\theta\varphi}C_\theta P_\alpha)f;\label{han1}\\\label{han2}
&(A_{\varphi}^{\theta, \alpha}C_{\alpha,\frac{\theta}{\alpha}}-C_{\alpha,\frac{\theta}{ \alpha}}A_{\overline{\varphi}}^{ \alpha,\theta}P_\alpha)f=(\widetilde{H}_{\alpha} H_{\varphi}C_{\theta}-\widetilde{H}_{\theta} H_{\varphi}C_{\alpha}P_\alpha)f;\\
    &(\tfrac{\theta}{\alpha}\ata_\varphi C_\theta-C_\theta\aat_{\bar\varphi}P_\alpha\tfrac{\bar\theta}{\bar\alpha})f=
    (\widetilde{H}_{\theta}H_{\varphi}C_{\frac{\theta}{\alpha}}P_{\frac{\theta}{\alpha}}-
    \widetilde{H}_{\frac{\theta}{\alpha}}H_{\varphi}C_\theta)f\label{han3}
\end{align}
 for $f\in \kdt$.
\end{theorem}

\begin{proof}   As in the proof of Theorem \ref{csymmetry}, it is enough to consider $f=f_{\alpha}+\alpha f_{\frac{\theta}{\alpha}}$, $f_{\alpha}\in K^\infty_{\alpha}$, $f_{\frac{\theta}{\alpha}}\in K^\infty_{\frac{\theta}{\alpha}}$.
To prove \eqref{han1} 
note that by Proposition \ref{rozklad} and by Proposition \ref{proj}, we have
 \begin{multline*}
 \ata_\varphi C_\theta f=\ata_\varphi(C_{\frac{\theta}{\alpha}}f_{\frac{\theta}{\alpha}}+\tfrac{\theta}{\alpha}C_\alpha f_\alpha)=P_\alpha(\varphi C_{\frac{\theta}{\alpha}}f_{\frac{\theta}{\alpha}})+P_\alpha(\varphi\tfrac{\theta}{\alpha}\alpha\bar z\bar f_\alpha)\\
 =P(\alpha P^-(\bar\alpha\varphi C_{\frac{\theta}{\alpha}}f_{\frac{\theta}{\alpha}}))+P_\alpha(\theta\varphi\bar z\bar f_\alpha)=\widetilde H_\alpha H_{\bar\alpha\varphi}C_{\frac{\theta}{\alpha}}f_{\frac{\theta}{\alpha}}+P_\alpha(\theta\varphi\bar z\bar f_\alpha)
 \end{multline*}
 and 
\begin{align*}
C_\theta\aat_{\bar\varphi}P_\alpha f&=C_\theta P_\theta(\bar\varphi f_\alpha)=P_\theta(\theta\varphi\bar z\bar f_\alpha)
\\&=P_\alpha(\theta\varphi \bar z\bar f_\alpha)+\alpha P_{\frac{\theta}{\alpha}}(\bar\alpha\varphi C_\theta f_\alpha)\\&=P_\alpha(\theta\varphi \bar z\bar f_\alpha)+\alpha P(\tfrac{\theta}{\alpha}P^-(\bar\theta\varphi C_\theta f_\alpha))\\&=P_\alpha(\theta\varphi \bar z\bar f_\alpha)+\alpha \widetilde H_{\frac{\theta}{\alpha}}H_{\bar\theta\varphi}C_\theta f_\alpha.
\end{align*}
 To prove \eqref{han2} note firstly that
$A_{\varphi}^{\theta, \alpha} =A_{\varphi}^{\alpha}P_\alpha+P_{\alpha}(\varphi{\alpha}P_{\frac{\theta}{\alpha}} (\bar{\alpha} I_{K_{\theta}}))$. So we have
     \begin{equation*} A_{\varphi}^{\alpha}{C_{\alpha,\frac{\theta}{ \alpha}}}_{|\kda}= A_{\varphi}^{\alpha}C_{\alpha}= C_{ \alpha}A_{\bar{\varphi}}^{\alpha}=C_{\alpha, \frac{\theta}{\alpha}}A_{\bar{\varphi}}^{ \alpha}\end{equation*} on $\kda$.
 On the other hand,
\begin{multline*}
     P_{\alpha}(\varphi{\alpha}P_{\frac{\theta}{\alpha}}(\bar{\alpha}C_{\alpha, \frac{\theta}{\alpha}}f))= P_{\alpha}(\varphi{\alpha}P_{\frac{\theta}{\alpha}}(\bar{\alpha}(\alpha \bar{z}\,\bar{f_{\alpha}}+\alpha C_{\frac{\theta }{\alpha}} (f_{\frac{\theta}{\alpha}}))))\\=
     P_{\alpha}(\varphi\alpha C_{\frac{\theta }{\alpha}} (f_{\frac{\theta}{\alpha}})) = 
     P \alpha P^- \varphi C_{\frac{\theta}{\alpha}}f_{\frac{\theta}{\alpha}}.
     \end{multline*}
    Thus, $A_{\varphi}^{\theta, \alpha}C_{\alpha,\frac{\theta}{\alpha}}= C_{\alpha,\frac{\theta}{\alpha}}A_{\bar{\varphi}}^{\alpha}+\widetilde{H}_{\alpha}H_{\varphi}
         C_{\frac{\theta}{\alpha}}\bar{\alpha}$.
     Analogously, $A_{\bar{\varphi}}^{\alpha,\theta}=A_{\bar{\varphi}}^{\alpha}+\alpha P_{\frac{\theta}{\alpha}}\bar{\alpha} \, \bar{\varphi} P_{\alpha}$ and
  \begin{align*}
P&\alpha P^-(\varphi C_{\frac{\theta}{\alpha}}f_{\frac{\theta}{\alpha}}) - C_{\alpha,\frac{\theta}{\alpha}}(\alpha P_{\frac{\theta}{\alpha}}(\bar{\alpha} \, \bar{\varphi} P_{\alpha}f))\\
&=   P\alpha P^-(\varphi C_{\frac{\theta}{\alpha}}f_{\frac{\theta}{\alpha}}) + P\alpha P^-(\varphi \tfrac{\theta}{\alpha}C_\alpha f_{\alpha})\\& -P\alpha P^-(\varphi \tfrac{\theta}{\alpha}C_\alpha f_{\alpha})
-\alpha C_{\frac{\theta}{\alpha}}( P_ \frac{\theta}{\alpha}(\bar{\alpha} \, \bar{\varphi} f_{\alpha})) \\ & =
P\alpha P^-(\varphi C_{\theta} f)-(P\alpha P^-\bar\alpha+\alpha P_{\frac{\theta}{\alpha}}\bar\alpha)(\theta\varphi C_\alpha f_\alpha)\\
&= \widetilde{H}_{\alpha}H_{\varphi} C_{\theta} f-P\theta P^-(\varphi C_\alpha f_\alpha),
\end{align*}
since $P\alpha P^-\bar\alpha+\alpha P_{\frac{\theta}{\alpha}}\bar\alpha= P_\theta =P\theta P^-\bar\theta$.
      Hence $ (A_{\varphi}^{\theta, \alpha}C_{\alpha,\frac{\theta}{\alpha}}- C_{\alpha,\frac{\theta}{\alpha}}A_{\bar{\varphi}}^{\alpha,\theta})f=
           (\widetilde{H}_{\alpha}H_{\varphi} C_{\theta}-\widetilde{H}_{\theta}H_{\varphi} C_{\alpha}P_{\alpha})f$ for $f\in \kdt$.

To show \eqref{han3} consider $g=g_{\frac{\theta}{\alpha}}+\tfrac{\theta}{\alpha}g_\alpha$, $g_{\alpha}\in K^\infty_{\alpha}$, $g_{\frac{\theta}{\alpha}}\in K^\infty_{\frac{\theta}{\alpha}}$. Then by Proposition \ref{rozklad}  we have
\begin{multline*}
  \tfrac{\theta}{\alpha}\ata_\varphi C_\theta g=\tfrac{\theta}{\alpha}\ata_\varphi (C_\alpha g_\alpha+\alpha C_{\frac{\theta}{\alpha}}g_{\frac{\theta}{\alpha}})=\tfrac{\theta}{\alpha}P_\alpha(\varphi C_\alpha g_\alpha)+\tfrac{\theta}{\alpha}P_\alpha(\varphi\alpha C_{\frac{\theta}{\alpha}}g_{\frac{\theta}{\alpha}})
 \\ =\tfrac{\theta}{\alpha}P_\alpha(\varphi \alpha\bar z\bar g_\alpha)+\tfrac{\theta}{\alpha}P_\alpha(\varphi\alpha C_{\frac{\theta}{\alpha}}g_{\frac{\theta}{\alpha}}),
\end{multline*}
and
\begin{multline*}
  C_\theta\aat_{\bar\varphi}P_\alpha(\tfrac{\bar\theta}{\bar\alpha}g_{\frac{\theta}{\alpha}}+g_\alpha)=C_\theta P_\theta\bar\varphi g_\alpha=P_\theta C_\theta\bar\varphi g_\alpha
  =(P_{\frac{\theta}{\alpha}}+\tfrac{\theta}{\alpha}P_\alpha \tfrac{\bar\theta}{\bar\alpha})(\theta\varphi\bar z\bar g_\alpha)\\=P_{\frac{\theta}{\alpha}}(\theta\varphi\bar z\bar g_\alpha)+\tfrac{\theta}{\alpha}P_{\alpha}(\alpha\varphi\bar z\bar g_\alpha).
\end{multline*}
Hence
\begin{align*}\label{a22}
   \tfrac{\theta}{\alpha}&\ata_\varphi C_\theta g-C_\theta\aat_{\bar\varphi}P_\alpha(\tfrac{\bar\theta}{\bar\alpha}g)= \tfrac{\theta}{\alpha}P_\alpha(\alpha\varphi C_{\frac{\theta}{\alpha}}g_{\frac{\theta}{\alpha}})-P_{\frac{\theta}{\alpha}}(\theta\varphi\bar z\bar g_\alpha)\\
   &=P\tfrac{\theta}{\alpha}P_\alpha(\alpha\varphi C_{\frac{\theta}{\alpha}}g_{\frac{\theta}{\alpha}})+P\tfrac{\theta}{\alpha}P^-(\alpha\varphi C_{\frac{\theta}{\alpha}}g_{\frac{\theta}{\alpha}})\\&-(P\tfrac{\theta}{\alpha}P^-(\alpha\varphi\bar z\bar g_\alpha)+P\tfrac{\theta}{\alpha}P^-(\alpha\varphi C_{\frac{\theta}{\alpha}}g_{\frac{\theta}{\alpha}}))
  \\&=P\tfrac{\theta}{\alpha}(P_\alpha+P^-)(\alpha\varphi C_{\frac{\theta}{\alpha}}g_{\frac{\theta}{\alpha}})-P\tfrac{\theta}{\alpha}P^-(\varphi ( C_\alpha g_\alpha+\alpha C_{\frac{\theta}{\alpha}}g_\frac{\theta}{\alpha}))\\&=P\theta P^-(\varphi C_{\frac{\theta}{\alpha}}g_{\frac{\theta}{\alpha}})-P\tfrac{\theta}{\alpha}P^-(\varphi C_\theta g)=\widetilde{H}_\theta H_\varphi C_{\frac{\theta}{\alpha}}g_{\frac{\theta}{\alpha}}-\widetilde{H}_{\frac{\theta}{\alpha}}H_\varphi C_\theta g,
\end{align*}
since by Proposition \ref{proj} $P_\alpha+P^-=\alpha P^-\bar\alpha$.
  \end{proof}

From \eqref{han3} we can obtain in particular the following:
\begin{corollary} Let $\alpha$, $\theta$ be nonconstant inner functions such that $\alpha\leqslant \theta$. If $A_{\varphi}^{\theta, \frac{\theta}{\alpha}}\in\mathcal{T}(\theta,\frac{\theta}{\alpha})$ for $\varphi\in L^2$, then
\begin{equation}\label{cos}\alpha A_{\varphi}^{\theta, \frac{\theta}{\alpha}}C_{\theta}- C_{\theta}A_{\bar{\varphi} }^{\frac{\theta}{\alpha}, \theta}P_{\frac\theta\alpha}\bar{\alpha}=
           \widetilde{H}_{\theta}H_{\varphi} C_{\alpha}P_\alpha -\widetilde{H}_{\alpha}H_{\varphi} C_{\theta}.\end{equation}
\end{corollary}
Note that comparing \eqref{han2} with \eqref{cos} we get:
\begin{equation}
 A_{\varphi}^{\theta, \alpha}C_{\alpha,\frac{\theta}{\alpha}}+\alpha A_{\varphi}^{\theta, \frac{\theta}{\alpha}}C_{\theta}=C_{\alpha,\frac{\theta}{\alpha}}A_{\bar{\varphi}}^{\alpha, \theta}P_\alpha+
 C_{\theta}A_{\bar{\varphi}}^{\frac{\theta}{\alpha}, \theta}P_{\frac\theta\alpha}\bar{\alpha},
\end{equation}
which is equivalent to \eqref{sym}. Hence we obtained another proof of \eqref{sym}.

\newpage

\section{Examples with Hankel matrices}
To illustrate the equalities in  Theorem \ref{h1} let us consider the following examples.

\begin{example}
Let $\alpha=z^3$, $\theta=z^5$ and $\varphi=\sum_{n=-4}^{2}a_nz^n\in \overline{K_{z^5}}+K_{z^3}$. Then for $f=(z_0,z_1,z_2, z_3, z_4)\in K_{z^5}$ we have, regarding the left hand side of \eqref{han1},
 \begin{equation*}A_{\varphi}^{z^5, z^3}C_{z^5}f=\left[ \   \begin{BMAT}{ccccc}{ccc}
    a_{-4} & a_{-3} & a_{-2} & a_{-1} & a_{0} \\
    a_{-3} & a_{-2} & a_{-1} & a_{0} & a_{1}\\
    a_{-2} & a_{-1} & a_{0} & a_1 & a_2 \addpath{(0,0,0)rrruuulllddd}
   \end{BMAT} \right]
\left[\begin{BMAT}{c}{ccccc} \bar z_0 \\ \bar z_1\\ \bar z_2\\ \bar z_3\\ \bar z_4 \end{BMAT} \right]\end{equation*}
 and
 \begin{equation*}
 C_{z^5}A_{\overline{z^5}}^{z^3, z^5}P_{z^3} f=\left[\    \begin{BMAT}{ccc}{ccccc}
    a_{-4} & a_{-3} & a_{-2} \\
    a_{-3} & a_{-2} & a_{-1} \\
    a_{-2} & a_{-1} & a_0 \\
    a_{-1} & a_0& a_ 1 \\
    a_0 & a_1 & a_2\addpath{(0,2,0)rrruuulllddd}
   \end{BMAT}\  \right]
\left[\begin{BMAT}{c}{ccc} \bar z_0 \\ \bar z_1\\ \bar z_2
\end{BMAT} \right].
 \end{equation*}
 The right hand side is given by Hankel matrices
 $$\widetilde{H}_{z^3}H_{\bar z^3z^5}C_{z^2}P_{z^2}(\bar z^3f)=
 \left[   \begin{BMAT}{cc}{ccc}
     a_{-1} & a_{0} \\
     a_{0} & a_{1}\\
     a_1 & a_2\\
   \end{BMAT} \right]
\left[\begin{BMAT}{c}{cc}  \bar z_3\\ \bar z_4 \end{BMAT} \right]
 $$
 and
 \begin{equation*}
z^3\widetilde{H}_{z^2}H_{\bar z^5\varphi}C_{z^5} P_{z^3} f=
\left[   \begin{BMAT}{ccc}{cc}
     a_{-1} & a_{0}& a_{1} \\
     a_{0} & a_{1}& a_{2}
   \end{BMAT} \right]
\left[\begin{BMAT}{c}{ccc}  \bar z_0 \\ \bar z_1\\ \bar z_2 \end{BMAT} \right].
 \end{equation*}
\end{example}
\begin{example} The equation \eqref{han2} will be illustrated with the same data as before. Hence
 \begin{equation*}A_{\varphi}^{z^5, z^3}C_{z^3, z^2}f=\left[  \  \begin{BMAT}{ccccc}{ccc}
    a_{-2} & a_{-1} & a_{0} & a_{-4} & a_{-3} \\
    a_{-1} & a_{0} & a_{1} & a_{-3} & a_{-2}\\
    a_{0} & a_{1} & a_{2} & a_{-2} & a_{-1}\addpath{(0,0,0)rrruuulllddd}
   \end{BMAT} \right]
\left[\begin{BMAT}{c}{ccccc} \bar z_0 \\ \bar z_1\\ \bar z_2\\ \bar z_3\\ \bar z_4 \end{BMAT} \right]
 \end{equation*}
 and
 \begin{equation*}
 C_{z^3,z^2}A_{\overline{\varphi}}^{z^3,z^5}P_{z^3} f=\left[ \   \begin{BMAT}{ccc}{ccccc}
    a_{-2} & a_{-1} & a_{0} \\
    a_{-1} & a_{0} & a_{1} \\
    a_{0} & a_{1} & a_2 \\
    a_{-4} & a_{-3}& a_{-2} \\
    a_{-3} & a_{-2} & a_{-1}\addpath{(0,2,0)rrruuulllddd}
   \end{BMAT}\  \right]
\left[\begin{BMAT}{c}{ccc} \bar z_0 \\ \bar z_1\\ \bar z_2
\end{BMAT} \right].
 \end{equation*}

 On the other hand,
 \begin{equation*}\widetilde{H}_{z^3}H_{\varphi} C_{z^5}f=\left[\    \begin{BMAT}{ccccc}{ccc}
    0 & 0 & 0 & a_{-4} & a_{-3} \\
    0 & 0& a_{-4} & a_{-3} & a_{-2}\\
    0 & a_{-4} & a_{-3} & a_{-2} & a_{-1}\addpath{(0,0,0)rrruuulllddd}
   \end{BMAT} \right]
\left[\begin{BMAT}{c}{ccccc} \bar z_0 \\ \bar z_1\\ \bar z_2\\ \bar z_3\\ \bar z_4 \end{BMAT} \right]
 \end{equation*}

 and
 \begin{equation*}
 \widetilde{H}_{z^5}H_{\varphi} C_{z^3}P_{z^3}f=\left[ \   \begin{BMAT}{ccc}{ccccc}
    0 & 0 & 0 \\
 0 & 0 & a_{-4} \\
    0 & a_{-4} & a_{-3} \\
    a_{-4} & a_{-3}& a_{-2} \\
    a_{-3} & a_{-2} & a_{-1}\addpath{(0,2,0)rrruuulllddd}
   \end{BMAT}\  \right]
\left[\begin{BMAT}{c}{ccc} \bar z_0 \\ \bar z_1\\ \bar z_2
\end{BMAT} \right].
 \end{equation*}
\end{example}

\begin{example}Using the same data again we obtain for the equation \eqref{han3}
\begin{equation*}z^2A_{\varphi}^{z^5, z^3}C_{z^5}f=z^2\left[   \begin{BMAT}{ccccc}{ccc}
    a_{-4} & a_{-3} & a_{-2} & a_{-1} & a_{0} \\
    a_{-3} & a_{-2} & a_{-1} & a_{0} & a_{1}\\
    a_{-2} & a_{-1} & a_{0} & a_1 & a_2\addpath{(2,0,0)rrruuulllddd}
   \end{BMAT}\  \right]
\left[\begin{BMAT}{c}{ccccc} \bar z_0 \\ \bar z_1\\ \bar z_2\\ \bar z_3\\ \bar z_4 \end{BMAT} \right]\end{equation*}
and
\begin{equation*}
 C_{z^5}A_{\overline{\varphi}}^{z^3,z^5}P_{z^3}\alpha \bar z^2f=\left[ \  \begin{BMAT}{ccc}{ccccc}
    a_{-4} & a_{-3} & a_{-2} \\
    a_{-3} & a_{-2} & a_{-1} \\
    a_{-2} & a_{-1} & a_0 \\
    a_{-1} & a_0& a_ 1 \\
    a_0 & a_1 & a_2\addpath{(0,0,0)rrruuulllddd}
   \end{BMAT}\ \right]
\left[\begin{BMAT}{c}{ccc} \bar z_2 \\ \bar z_3\\ \bar z_4
\end{BMAT} \right].
 \end{equation*}
 On the other hand,
 \begin{equation*}
 \widetilde{H}_{z^5}H_{\varphi} C_{z^2}P_{z^2}f=\left[ \  \begin{BMAT}{cc}{ccccc}
    0 & 0 \\
 0 & a_{-4} \\
  a_{-4} & a_{-3} \\
 a_{-3}& a_{-2} \\
 a_{-2} & a_{-1}\addpath{(0,3,0)rruulldd}
   \end{BMAT}\ \right]
\left[\begin{BMAT}{c}{cc} \bar z_0 \\ \bar z_1
\end{BMAT} \right]
 \end{equation*}
 and
 \begin{equation*}\widetilde{H}_{z^2}H_{\varphi} C_{z^5}f=\left[\    \begin{BMAT}{ccccc}{cc}
    0 & 0& a_{-4} & a_{-3} & a_{-2}\\
    0 & a_{-4} & a_{-3} & a_{-2} & a_{-1}\addpath{(0,0,0)rruulldd}
   \end{BMAT} \right]
\left[\begin{BMAT}{c}{ccccc} \bar z_0 \\ \bar z_1\\ \bar z_2\\ \bar z_3\\ \bar z_4 \end{BMAT} \right].
 \end{equation*}
\end{example}

\end{document}